\documentclass[a4paper]{amsart}

\input{custom-commands}
\usepackage{todonotes}
\usepackage{standalone}

\title{Thompson's group $T$ has quadratic Dehn function}
\author{Matteo Migliorini}


\begin{document}

\NewDocumentCommand{\tree}{o}{\IfNoValueTF{#1}{\mathcal T}{\mathcal T^{(#1)}}}
\NewDocumentCommand{\Fn}{o}{\IfNoValueTF{#1}{F_n}{F_n^{(#1)}}}

\begin{abstract}
	We prove that Thompson's group $T$ and, more generally, all the Higman--Thompson groups $T_n$ have quadratic Dehn function.
\end{abstract}

\maketitle

\section{Introduction}

In 1965 Richard Thompson defined the three groups $ F < T < V $ that are now known as Thompson's groups.
They have been studied extensively in the last fifty years, as they arise naturally in many different areas of mathematics (geometric group theory, homotopy theory, to name a few).
For a historical on Thompson's groups, we refer to the excellent survey by Cannon, Floyd, and Parry \cite{CannonFloydParry}.

Different points of view mean that Thompson's groups can be defined in many different ways. For example, one can define them combinatorially, as groups of certain partial automorphisms of a binary tree, or dynamically, as subgroups of homeomorphisms (with possibly some discontinuities, in the case of $V$) from the circle to itself.

Despite receiving a lot of attention, many questions about these groups are still open, the most famous being arguably whether $F$ is amenable.

Being finitely presented groups, an interesting problem is computing their Dehn function. Indeed, the Dehn function is a natural quasi-isometric invariant associated to finitely presented groups; geometrically, it can be seen as the isoperimetric function associated to the presentation complex of the group.

Equivalently, one can also interpret the Dehn function combinatorially as the number of relations one needs to apply to reduce a word of length $n$ representing the identity to the empty word. In this sense, it is also related to the complexity of the word problem; for example, the latter is decidable if and only if the Dehn function is computable.

Note that the Dehn function of a finitely presented group is always either linear or at least quadratic; having linear Dehn function is an equivalent definition of the notion of hyperbolicity, introduced by Gromov in 1987.

The Dehn function of Thompson's group $F$ was proven by Guba \cite{Guba} to be quadratic. While the lower bound follows easily from the fact that $F$ contains $ \ZZ^2 $ as a subgroup and is therefore not hyperbolic, the upper bound required a lot of work.
Gersten first managed to prove an exponential upper bound \cite{GerstenExponential}, that was brought down to sub-exponential \cite{GubaSapir}, then to $n^5$ \cite{GubaPolynomialF}, and finally to quadratic \cite{Guba}.

Thompson's groups $T$ and $V$ are also not hyperbolic, so their Dehn function is at least quadratic. In fact, both are conjectured to have quadratic Dehn function (see for example \cite{ZaremskyProblems}).
It was proven by Guba \cite{GubaPolynomialFTV} that the Dehn functions of $T$ and $V$ have polynomial upper bounds of $ n^7 $ and $ n^{11} $ respectively. The upper bound for $T$ was later improved to $n^5$ \cite{WangZhengZhang}.

In this paper, we prove the conjecture for Thompson's group $T$.

\begin{maintheorem}\label{t-dehn-quadratic}
	The Dehn function of Thompson's group $T$ is quadratic.
\end{maintheorem}

The proof relies on the fact that $F$ has quadratic Dehn function.

If in the combinatorial definition of Thompson's groups one replaces the binary tree with an $n$-ary tree, one gets three families of groups $ F_n, T_n, V_n $. The family $V_n$ was introduced by Higman \cite{Higman}, while later Brown \cite{BrownFiniteness} naturally extended the definition to $F_n$ and $T_n$; these groups are sometimes referred to as \emph{Higman--Thompson groups}.

The same techniques used for proving \cref{t-dehn-quadratic} apply also for the family $ T_n $.

\begin{maintheorem}\label{tn-dehn-quadratic}
	The Dehn function of Higman--Thompson group $T_n$ is quadratic.
\end{maintheorem}

Here we exploit the fact that by a result of Zhang \cite{Zhang} the groups $F_n$ have quadratic Dehn function.

\subsection{Strategy of the proof}

Thompson's groups $F$ and $T$ both have a finite and an infinite generating set that are well-established in the literature; the advantage of the infinite one is that it admits a more symmetric (infinite) presentation.

Both groups admit normal forms, sometimes called $ pq $ and $ pcq $ factorization (see \cite[Corollary 2.7, Theorem 5.7]{CannonFloydParry}), both of which use the infinite presentation; Guba used this normal form to prove that $F$ has quadratic Dehn function. For $T$ the situation is slightly more complicated: if we rewrite the normal form of an element of $T$ in terms of a finite generating set, the resulting word may be much longer than the actual norm of the element with respect to the word metric.
Therefore, we introduce two different normal forms for $F$ and $T$ that use finite generating sets.

We look at Thompson's groups from the dynamical point of view, so $T$ is viewed as a certain subgroup of orientation-preserving self-homeomorphisms of the circle. If we parametrize the circle as the interval $[0,1]$ with the endpoints identified, then $F$ can be described as the subgroup of elements of $T$ that fix the point $0 \equiv 1$. By considering the subgroup of elements that fix both $0$ and $\frac12$, we get a subgroup $H < F$ that is naturally isomorphic to $ F \times F $.

We choose a finite generating set for $H$. From this we construct a finite generating set for $F$ by adding an element $X_0 \in F$ that does not fix $ \frac 12 $. Similarly, from the generating set for $F$ we construct a generating set for $T$ by adding the rotation of the circle $C_0 \in T$ of order two.
By studying the dynamics of these groups $ H < F < T $ on the circle, we can prove that every element of $F$ is represented by a word with at most one instance of $X_0$, and every element of $T$ is represented by a word with at most one instance of $C_0$.
Such words are our chosen normal forms for $F$ and $T$; these are not unique, but this is not an issue for our purposes.

By a standard argument, it suffices to show that all triangular loops in the Cayley graph where the edges are normal forms have quadratic area. To perform this computation we exploit the fact that in a triangular loop almost all the letters are generators of $H$. Since $C_0$ stabilizes $ \set{0, \frac12} $, it normalizes $ H $; in fact, the generating set chosen for $H$ is invariant under conjugation by $C_0$. So we can commute the letter $C_0$ with every generator of $H$, by replacing the latter with its conjugate. This allows us to eliminate all instances of $C_0$ in the word, and then we conclude using that the Dehn function of $F$ is quadratic.

For \cref{tn-dehn-quadratic} the strategy is similar, except now we consider the subgroup of $F_n$ that fixes all the points $ \frac 1n, \dots, \frac{n-1}n $, and we work with the rotation of order $n$.

Note that one can prove \cref{tn-dehn-quadratic} directly without proving the particular case for Thompson's group $ T = T_2 $. However, we choose to first give the proof for \cref{t-dehn-quadratic} as the reader might be more familiar with Thompson's groups $F$ and $T$ than their generalizations, and the former have more established conventions for their generating set. This makes the proof of \cref{t-dehn-quadratic} more explicit and hopefully easier to read, and in turn it helps to understand the proof of \cref{tn-dehn-quadratic}.

\subsection{Structure of the paper}

We start in \cref{section-preliminaries} by recalling the definition of Dehn function, and we introduce Thompson's group $F$ and $T$, together with their generating set, presentation, and some useful properties. In \cref{section-f} we show that $F$ admits a normal form with at most one instance of the generator $X_0$, and
in \cref{section-t} we introduce the normal form for $T$, that we use to prove \cref{t-dehn-quadratic}.
In \cref{section-tn}, we generalize all the arguments in the previous sections to prove \cref{tn-dehn-quadratic}.
We conclude with some remarks in \cref{section-remarks}.

\subsection{Acknowledgments}
We would like to thank Claudio Llosa Isenrich for pointing us to this question, and for useful discussions. The author gratefully acknowledges funding by the DFG 281869850 (RTG 2229), and would like to thank INdAM as a member of GNSAGA.

\section{Preliminaries}\label{section-preliminaries}

We start by fixing some notation. Let $G$ be a finitely generated group, and $S=\set{s_1, \dots, s_k}$ be a finite subset of $G$. If $w$ is a word in the alphabet $S \cup S^{-1}$, then we write $ w(S) $ or $ w(s_1, \dots, s_k) $ to make the alphabet explicit. We denote the length of $w$ by $ \abs w $. If $g \in G$, and $S$ is a generating set, we denote by $ \norm{g}_S $ the minimal length of a word representing $g$ (that is, the distance of $g$ from the identity with respect to the word metric). If the generating set is clear, we also write $ \norm g_G $.

We say that a word $w(S)$ is null-homotopic if it represents the trivial element of $G$. If $G$ has a finite presentation $ \presentation S{\mathcal R} $, and $w$ is a null-homotopic word, then $w$ can be written as a product of conjugates of $k$ relations, for some $ k \in \NN $; we call the minimum such $k$ the \emph{area} of $w$, and denote it by $ \Area(w) $.

The Dehn function of $G$ encodes the growth of the area in terms of the length of the word. More precisely, it is a function $ \Dehn G \colon \NN \to \NN $ defined by
\[
	\Dehn G (N) = \max \set{\Area(w) : w \text{ null-homotopic},\,\abs{w} \leq N}.
\]

The Dehn function as defined above depends on the chosen presentation for $G$; to solve this problem, one introduces the following equivalence relation.

If $ f,g \colon \NN \to \NN $, we say that $ f \asympleq g $ if there exists $ C>0 $ such that for all $ N \in \NN  $ we have $  f(N) \leq C g(CN+N) + CN + C  $. If $ f \asympleq g \asympleq f $ we write $ f \asymp g $; different presentations (and more in general quasi-isometric groups) have equivalent Dehn functions.
A group $G$ has quadratic Dehn function if $ \Dehn G \asymp (N \mapsto N^2) $.

Throughout the paper, we identify the circle with the interval $[0,1]$ with the endpoints glued together. Thompson's group $T$ is the group of self-homeomorphisms of the circle that send dyadic rationals to dyadic rationals, are linear except on finitely many dyadic rationals, and the derivative is a power of $2$ (with integer exponent) wherever it is defined.
The subgroup of $T$ consisting of the homeomorphisms that fix the point $0$ is Thompson's group $F$. We refer to \cite{CannonFloydParry} for an introduction on these groups.

Let $ X_0, X_1, \dots $ denote the standard generators of Thompson's group $F$. They are defined as follows:
\[
	X_k(x) = \begin{cases}
		x                                           & \text{if } 0 \leq x \leq 1-\frac1{2^k},                     \\
		\frac12 \cdot x + \frac12 - \frac1{2^{k+1}} & \text{if } 1-\frac1{2^k}\leq x \leq 1-\frac{1}{2^{k+1}},    \\
		x-\frac1{2^{k+2}}                           & \text{if } 1-\frac1{2^{k+1}} \leq x \leq 1-\frac1{2^{k+2}}, \\
		2x-1                                        & \text{if } 1-\frac1{2^{k+2}} \leq x \leq 1.
	\end{cases}
\]
Thompson's group $F$ admits both an infinite presentation involving all the $X_i$, with $i \in \NN$, and a finite one with two generators $ X_0, X_1 $, given by
\[
	\presentation{X_0,X_1}{
	[X_0X_1^{-1}, X_0^{-1}X_1X_0],
	[X_0X_1^{-1}, X_0^{-2}X_1X_0^2]
	}
\]

In particular, all the other generators can be expressed in terms of the $X_i$ by the formula $ X_{i+1} = X_0^{-i} X_1 X_0^i $.

To obtain a generating set for $T$ we need to add another homeomorphism that does not fix $0$. To this end, let $C_0$ be the homeomorphism defined by
\[
	C_0(x) = \begin{cases}
		x+\frac12 & \text{if } 0 \leq x \leq \frac12  \\
		x-\frac12 & \text{if } \frac12 \leq x \leq 1. \\
	\end{cases}
\]
In other words, $C_0$ is the rotation of the circle of order two.

\begin{remark}
	One can also define elements $C_i$ of order $ i+2 $ for all $ i \geq 0 $, as done in \cite{CannonFloydParry}, and get an infinite presentation. To obtain a finite generating set, it suffices to add only one of the $C_i$ to the generators of $F$, and the usual choice is to use $C_1$ instead of $C_0$. To translate between the presentations found in the literature and the one we use here, it is helpful to keep in mind that $ C_1 = X_0^{-1} C_0 $.
\end{remark}

We also define $ Y_i := C_0 X_i C_0 $ for $ i\geq 1 $. Since $X_i$ has $ \frac12 $ as a fixed point for $ i \geq 1 $, it follows that $ Y_i \in F $.
In this paper, we choose $ S_F := \set{X_0^{\pm 1}, X_1^{\pm 1}, X_2^{\pm 1}, Y_1^{\pm 1}, Y_2^{\pm 1}} $ as a generating set for $F$, and $S_T:= S_F \cup \set{C_0}$ as a generating set for $T$. In \cref{figure-rectangular-diagrams} we have drawn the rectangular diagrams for the generators: the top border represents the domain, while the bottom border represents the codomain (see \cite{CannonFloydParry} for more details).

\begin{figure}
	\begin{tikzpicture}[scale=3]
	\def\he{0.3}
	\def\x{0}
	\def\y{0}
	\draw (\x,\y) rectangle (\x+1,\y+\he);
	\draw (\x+0.5,\he+\y) -- (\x+0.25, \y);
	\draw (\x+0.75,\he+\y) -- (\x+0.5, \y);
	\node[below] at (\x+0.5,\y) {$X_0$};
	\node[above] at (\x,\y+\he) {$0$};
	\node[above] at (\x+1,\y+\he) {$1$};
	\node[above] at (\x+1/2,\y+\he) {$\frac 12$};
	\node[above] at (\x+3/4,\y+\he) {$\frac 34$};

	\def\x{1.5}
	\draw (\x,\y) rectangle (\x+1,\y+\he);
	\draw (\x+0.5,\he+\y) -- (\x+0.5, \y);
	\draw (\x+0.75,\he+\y) -- (\x+5/8, \y);
	\draw (\x+7/8,\he+\y) -- (\x+3/4, \y);
	\node[below] at (\x+0.5,\y) {$X_1$};
	\node[above] at (\x,\y+\he) {$0$};
	\node[above] at (\x+1,\y+\he) {$1$};
	\node[above] at (\x+1/2,\y+\he) {$\frac 12$};
	\node[above] at (\x+3/4,\y+\he) {$\frac 34$};
	\node[above] at (\x+7/8,\y+\he) {$\frac 78$};

	\def\x{3}
	\draw (\x,\y) rectangle (\x+1,\y+\he);
	\draw (\x+3/4,\he+\y) -- (\x+3/4, \y);
	\draw (\x+7/8,\he+\y) -- (\x+13/16, \y);
	\draw (\x+15/16,\he+\y) -- (\x+7/8, \y);
	\node[below] at (\x+0.5,\y) {$X_2$};
	\node[above] at (\x,\y+\he) {$0$};
	\node[above] at (\x+1,\y+\he) {$1$};
	\node[above] at (\x+3/4,\y+\he) {$\frac 34$};
	\node[above] at (\x+7/8,\y+\he) {$\frac 78$};

	\def\x{0}
	\def\y{-.7}
	\draw (\x,\y) rectangle (\x+1,\y+\he);
	\draw (\x+1/2,\he+\y) -- (\x, \y);
	\draw (\x+1,\he+\y) -- (\x+1/2, \y);
	\node[below] at (\x+0.5,\y) {$C_0$};
	\node[above] at (\x,\y+\he) {$0$};
	\node[above] at (\x+1,\y+\he) {$1$};
	\node[above] at (\x+1/2,\y+\he) {$\frac 12$};

	\def\x{1.5}
	\draw (\x,\y) rectangle (\x+1,\y+\he);
	\draw (\x+1/2,\he+\y) -- (\x+0.5, \y);
	\draw (\x+1/4,\he+\y) -- (\x+1/8, \y);
	\draw (\x+3/8,\he+\y) -- (\x+1/4, \y);
	\node[below] at (\x+0.5,\y) {$Y_1$};
	\node[above] at (\x,\y+\he) {$0$};
	\node[above] at (\x+1,\y+\he) {$1$};
	\node[above] at (\x+1/2,\y+\he) {$\frac 12$};
	\node[above] at (\x+1/4,\y+\he) {$\frac 14$};
	\node[above] at (\x+3/8,\y+\he) {$\frac 38$};

	\def\x{3}
	\draw (\x,\y) rectangle (\x+1,\y+\he);
	\draw (\x+1/2,\he+\y) -- (\x+0.5, \y);
	\draw (\x+1/4,\he+\y) -- (\x+1/4, \y);
	\draw (\x+3/8,\he+\y) -- (\x+5/16, \y);
	\draw (\x+7/16,\he+\y) -- (\x+3/8, \y);
	\node[below] at (\x+0.5,\y) {$Y_2$};
	\node[above] at (\x,\y+\he) {$0$};
	\node[above] at (\x+1,\y+\he) {$1$};
	\node[above] at (\x+1/2,\y+\he) {$\frac 12$};
	\node[above] at (\x+1/4,\y+\he) {$\frac 14$};
	\node[above] at (\x+3/8,\y+\he) {$\frac 38$};

\end{tikzpicture}
	\caption{Rectangular diagrams for the generators $ S_T $.}
	\label{figure-rectangular-diagrams}
\end{figure}
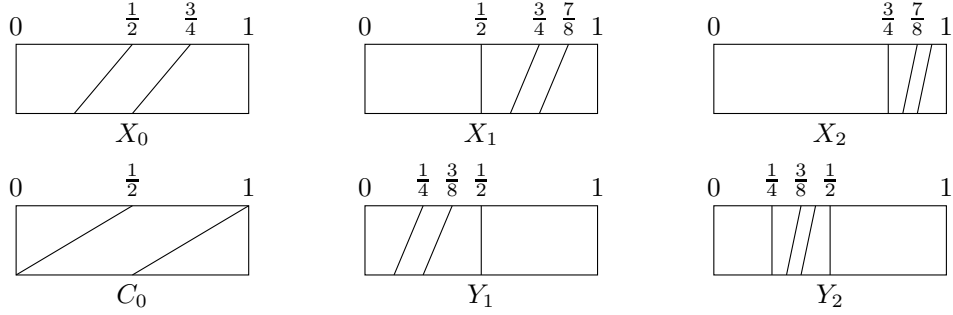

\begin{lemma}
	Thompson's group $F$ admits the following presentation:
	\[
		\presentation{S_F}{
		[Y_1, X_1],
		[Y_1, X_2],
		X_0 X_2=X_1 X_0,
		Y_1 = X_0^2 X_1^{-1} X_0^{-1},
		Y_{2} = X_0 X_1^2 X_2^{-1} X_1^{-1} X_0^{-1}
		}
	\]
\end{lemma}

\begin{proof}
	By using the rectangular diagrams, one can check that the relations hold in $F$. Moreover, one can use the last three relations to express $ X_2, Y_1, Y_2 $ in terms of $X_0, X_1$; after doing this, it turns out that the remaining two relations are conjugate to $ [X_0X_1^{-1}, X_0^{-1}X_1X_0], [X_0X_1^{-1}, X_0^{-2}X_1X_0^2]$, so the above presentation is equivalent to the standard presentation for $F$.
\end{proof}

The reason why we introduced these additional generators is that $\set {X_1,X_2} $ and $ \set{Y_1,Y_2} $ generate two subgroups of $F$ that are both isomorphic to $F$ itself. More generally, given two dyadic numbers $ 0 \leq a < b \leq 1 $ with $ b-a=2^{-k} $ we may consider the subgroup of $F$ consisting of homeomorphism that are the identity outside $ (a,b) $; denote this group by $ F[a,b] $.

\begin{theorem}\cite[Theorem 3.1.3, 5.1.1]{BurilloIntro}\label{undistorted-rescaling}
	For every dyadic numbers $ 0 \leq a < b \leq 1 $, the subgroup $ F[a,b] $ is isomorphic to $F$ and undistorted.
\end{theorem}

In particular $ X_1, X_2 $ generate $ F[\frac12, 1] $, while $Y_1,Y_2$ generate $ F[0, \frac12] $; the element $C_0 \in T$ conjugates $ F[0, \frac12] $ to $ F[\frac12, 1] $ and vice versa, as we can see from the following.

\begin{lemma}
	Thompson's group $T$ is generated by the elements $ S_T $ and satisfies the following relations:
	\begin{enumerate}
		\item $[Y_1, X_1]$,\label{relation-f-1}
		\item $[Y_1, X_2]$,\label{relation-f-2}
		\item $X_0 X_2=X_1 X_0$,\label{relation-f-3}
		\item $Y_1 = X_0^2 X_1^{-1} X_0^{-1}$,\label{relation-f-4}
		\item $ X_0^{-1} Y_2 X_0 = X_1^2 X_2^{-1} X_1^{-1} $,\label{relation-f-5}
		\item $ C_0^2 = 1 $, \label{relation-involution}
		\item $ (C_0 X_0)^3 = 1 $, \label{relation-order-three}
		\item $ C_0 X_2 = Y_2 C_0 $, \label{relation-shuffle-one}
		\item $ C_0 X_1 = Y_1 C_0 $.\label{relation-shuffle-two}
	\end{enumerate}
\end{lemma}

\begin{proof}
	The first $5$ relations are the relations from $F$, while the last $4$ can be checked directly by composing the homeomorphisms.
\end{proof}

\begin{remark}
	It is possible to show that the above relations yield a presentation for $T$, as they can be obtained by manipulating the finite presentations for $T$ found in the literature \cite{CannonFloydParry,BurilloClearySteinTaback}, but showing this equivalence by hand is a bit tedious. Luckily we can avoid doing this computation: while proving that the Dehn function of $T$ is quadratic, we actually show that for some constant $C>0$ any null-homotopic word of length $n$ in the generating set $S_T$ can be reduced to the trivial word by applying at most $ C \cdot n^2 $ relations (1)-(9), which implies that the above is a presentation for $T$.
\end{remark}

\section{A normal form for F}\label{section-f}

Except for $X_0$, the generators for $F$ interact quite nicely with $C_0$, as $ C_0 X_i = Y_i C_0 $ for $ i \geq 1 $. The aim of this section is to show that every element of $F$ can be expressed as a word in the generators $S_F$ with at most one instance of $X_0$ (or its inverse). This defines a normal form for $F$, which will be very useful when computing the Dehn function of $T$.

Before doing so, we need some estimates on the word metric of $F$. An efficient way to estimate the word length obtained by looking at the so-called reduced tree diagrams for elements in $F$ \cite{BurilloSubgroups}. We translate this notion in terms of self-homeomorphisms of the interval.

Given $f \in F$, there exist subdivisions $ 0=x_0 < \dots < x_k=1 $ and $ 0 = y_0 < \dots < y_k = 1 $, with every interval $ [x_i, x_{i+1}] $ and $ [y_i, y_{i+1}] $ of the form $ [\frac{a}{2^{\ell}}, \frac{a+1}{2^\ell}] $, such that $f$ sends $ x_i $ to $ y_i $ and is affine on every $ [x_i, x_{i+1}] $. If the subdivision is minimal, then the number of intervals $k$ is quasi-equivalent to the norm $ \norm f_{F} $, that is $ \frac 1Ck - C \leq \norm f_{F} \leq Ck + C  $ for some constant $C > 0$ independent of $f$.

From this, we get the following lemmas.

\begin{lemma}\label{estimate-word-length-by-image}
	There exists a constant $C>0$ such that the following holds.
	Let $ 0<a=\frac{d}{2^k}<1 $ be a dyadic number, with $d$ odd. There exists an element $ f\in F $ such that $ f(a) = \frac12 $ and $ \norm f_{F} \leq C \cdot k + C$. Moreover, all $ f \in F $ with $ f(a) = \frac12 $ satisfy $ \norm{f}_{F} \geq \frac1C k-C $.
\end{lemma}

\begin{proof}
	There is a subdivision $ x_0 < \dots < x_{k+1} $ where $ x_j = a $ for some $j$; this can be constructed by iteratively bisecting the interval containing $a$. Then choose a subdivision $ y_0 < \dots < y_{k+1} $ such that $y_j = \frac12$; the map associated to these subdivisions satisfies the requirements.

	For the second statement, let $ f \in F $ be any map with $ f(a)=\frac12 $, and choose subdivisions associated to $f$. Since $ \frac 12 $ is a point of the second subdivision (unless $f$ is the identity), then $ a $ must be a point of the first subdivision, and so the two subdivisions must have at least $k+1$ intervals; from this we conclude.
\end{proof}

\begin{lemma}\label{stabilize-one-half}
	There exists $C>0$ such that the following holds.
	Let $ 0<a=\frac{d}{2^k} <1 $ be a dyadic number, with $ a \neq \frac12 $ and $d$ odd. There exists an element $ f\in F $ with $ f(a) = \frac12 $ that can be represented as a word of length at most $ C \cdot k $ of the form:
	\begin{align*}
		X_0      & \cdot w(X_1,X_2)  &  & \text{if } a>\tfrac12    \\
		X_0^{-1} & \cdot w(Y_1, Y_2) &  & \text{if } a < \tfrac12.
	\end{align*}
\end{lemma}

\begin{proof}
	If $ a > \frac12 $, by \cref{estimate-word-length-by-image} we can obtain an element of $F$ that sends $2a-1$ to $ \frac12 $: by rescaling with $ x \mapsto \frac12(x+1) $ we get $ g \in F[\frac12, 1] $ with $ \norm{g}_{F[\frac12, 1]} \leq C(k-1) $ and $ g(a) = \frac34 $. If $w(X_1,X_2)$ is a shortest representative for $g$, then $ X_0 w(X_1,X_2) $ satisfies the hypotheses.

	If $ a<\frac12 $, the construction is analogous by choosing $ g \in F[0, \frac12] $ with $g(a)=\frac14$, and using $ X_0^{-1} w(Y_1,Y_2) $.
\end{proof}

We are now able to show that elements of $F$ can be represented by words that have at most one instance of the letter $ X_0 $.

\begin{proposition}\label{bound-nice-f}
	Let $ f \in F $. Then $f$ is represented by a word $w(S)$ of the form
	\begin{align*}
		w & = v(X_1,X_2) \cdot  X_0^{-1} \cdot  u(Y_1,Y_2) &  & \text{if }  f^{-1}(\tfrac12) < \tfrac12   \\
		w & = u(Y_1, Y_2) \cdot  X_0 \cdot  v(X_1,X_2)     &  & \text{if }  f^{-1}(\tfrac12) > \tfrac12   \\
		w & = v(X_1,X_2) \cdot  u(Y_1, Y_2)                &  & \text{if }  f^{-1}(\tfrac12) = \tfrac12 .
	\end{align*}
	Moreover, the word $w$ can be chosen of length bounded above by $ C \cdot \norm{f}_F $, for some constant $C$ independent of $f$.
\end{proposition}

\begin{proof}
	Let $ a=\frac d{2^k} \coloneqq f^{-1}(\frac12) $; assume that $ a > \frac12 $. Let $f_1$ be the function with $f_1(a) = \frac12$ given by \cref{stabilize-one-half}, represented by the word $ X_0 w_1(X_1,X_2) $ of length at most $ C \cdot k$. By \cref{estimate-word-length-by-image}, this means that the length of $w_1$ is bounded linearly in terms of $ \norm f_{F} $.

	Since $ f_1 \circ f^{-1} $ fixes $ \frac12 $, then it belongs to the subgroup $ F[0,\frac12] \times F[\frac12, 1] $, which is undistorted by \cref{undistorted-rescaling}. There exists $ f_2 \in F[0,\frac12] $, $ f_3 \in F[\frac12, 1] $, such that $ f_1 \circ f^{-1} = f_3 \circ f_2 $; let $ w_2(Y_1,Y_2) $ and $w_3(X_1,X_2)$ be words representing $ f_2 $ and $f_3 $,
	with
	\[
		\abs{w_2} + \abs{w_3} \leq C' \cdot \norm{f_1 \circ f^{-1}}_F \leq C' \cdot (\norm{f}_F + \norm{f_1}_F),
	\]
	so the length of both $ w_2 $ and $ w_3 $ is bounded by a linear function of $ \norm f_{F} $.

	So $ f = f_2^{-1} \circ f_3^{-1} \circ f_1$ is represented by
	\[
		w_2^{-1} (Y_1,Y_2) \cdot  w_3^{-1} (X_1, X_2) \cdot  X_0 \cdot w_1(X_1,X_2).
	\]
	Finally, since for $ \epsilon \in \set{\pm1} $ we have
	\[
		X_1^{\epsilon} X_0 = X_0 X_2^{\epsilon}, \qquad X_2^{\epsilon} X_0 = X_0 X_3^{\epsilon} = X_0 X_1^{-1} X_2^{\epsilon} X_1
	\]
	we obtain that $ w_3^{-1}(X_1,X_2) X_0 = X_0 w_4(X_1,X_2) $ for some word $w_4$. Therefore $f$ is represented by
	\[
		w_2^{-1}(Y_1,Y_2) \cdot  X_0  \cdot w_4(X_1,X_2),
	\]
	with $w_4$ bounded by a linear function of $\norm f_F$.

	If $f^{-1}(\frac12) < \frac12$, then $ f(\frac12) > \frac12 $, so the result follows by applying the above on $ f^{-1} $, and then taking the inverse at the end.

	If $f(\frac12)=\frac12$, then we conclude by noting that $f \in F[0,\frac12] \times F[\frac12,1]$ and by using that the latter is an undistorted subgroup of $F$ by \cref{undistorted-rescaling}
\end{proof}

\section{The Dehn function of $T$}\label{section-t}

The estimates for the word length based on tree diagrams hold for Thompson's group $T$ as well (see \cite{BurilloClearySteinTaback}). In particular, by the same arguments we get the following.

\begin{lemma}\label{lower-bound-norm-t-by-depth}
	There exists $ C>0 $ such that for every $ f \in T $, if $ f(0) = \frac{d}{2^k} $ with $d$ odd, then $ \norm f_{T} \geq \frac 1C k - C  $.
\end{lemma}

\begin{proposition}\label{nice-t}
	There exists $ C>0 $ such that every $ f \in T \setminus F $ is represented by a word
	\[
		w = u(S_F) \cdot C_0 \cdot v(S_F),
	\]
	such that $\abs w \leq C \cdot \norm{f}_T $.
\end{proposition}

\begin{proof}

	Let $ a= \frac{d}{2^k} := f^{-1}(0) $. Let $ f_1 \in F $ be the element that sends $ a $ to $ \frac12 $ given by \cref{estimate-word-length-by-image}. By \cref{lower-bound-norm-t-by-depth}, and since $F$ is undistorted in $T$ \cite[Corollary 5.2]{BurilloClearySteinTaback}, we have $\norm{f_1}_F \leq C' \norm f_{T} + C'$ for some constant $C'$.

	Now $ f_2 := f \circ f_1^{-1} \circ C_0 $ fixes the point $0$, so it belongs to $F$; moreover
	\[
		\norm{f_2}_F \leq C'' \norm{f_2}_T \leq C'' (\norm{f}_T + \norm{f_1}_T + 1).
	\]

	Since $ f = f_2 \circ C_0 \circ f_1 $, we may conclude by choosing short representatives $ w_1(S_F), w_2(S_F) $ for $ f_1, f_2 $.
\end{proof}

We have shown that all elements in $T$ are represented by a word containing at most one letter $C_0$, and whose length is bounded by a linear function of the norm of the element; we can employ this nice representatives it to compute the Dehn function of $T$.

A standard argument, employed for example in \cite{GerstenShort,CarterForester,SPF} shows that it suffices to bound the area of \emph{triangular loops}. This fact can be stated generally as follows:

\begin{proposition}\label{triangles}
	Let $G = \presentation S{\mathcal R}$ be a finitely presented group, and for each $g \in G$ choose a subset $W_g$ of words in $S$ that represent the element $g$.
	Suppose that there exists $C>0$ such the following hold:
	\begin{itemize}
		\item  for every $g \in G$ there is $ w \in W_g $ with $ \abs{w} \leq C \cdot \norm{g}_G $;
		\item for every null-homotopic word $ w = w_1 w_2 w_3 $, with $w_i \in W_{g_i}$ for some elements $ g_1, g_2, g_3 $, we have that $ \Area(w) \leq C \abs{w}^2 $.
	\end{itemize}
	Then the Dehn function of $G$ is quadratic.
\end{proposition}

The set $ W_g $ should be thought as some preferred word representatives for $g$. In our setting, the set $ W_g $ consists of the word representatives for $g \in T$ that have at most one $C_0$; they exist by \cref{nice-t}.

The proposition can be proven by subdividing a Van Kampen diagram for a generic loop into triangular cells, as explained in \cite{CarterForester}, or even directly by induction \cite[Theorem 5.9]{SPF}.
Note that it is non-restrictive to prove the lemma for the case $ \abs{W_g} = 1 $, that is when using a unique normal form, as restricting every $W_g$ to a singleton can only weaken the hypotheses.

This simplifies our work considerably, as it suffices to compute the area of null-homotopic words containing up to three instances of $C_0$. Note that a word containing exactly one instance of $C_0$ cannot be null-homotopic, as it cannot send the point $0$ to itself. So there are just two cases to check.

\begin{proposition}
	There exists $ C>0 $ such that every null-homotopic word of the form
	\[
		w=C_0 \cdot w_1(S_F) \cdot C_0 \cdot w_2(S_F)
	\]
	has area at most $ C \cdot \abs{w}^2 $.
\end{proposition}

\begin{proof}
	Let $ f_2 \in F $ be the element represented by $ w_2 $ and $ C_0 w_1^{-1} C_0 $. Since $ C_0 w_1 C_0 $ fixes the point $ \frac12 $, then $ f_2 $ can be represented by a word $ u(X_1,X_2,Y_1,Y_2)$ of length bounded by $ C \abs{w_2} $, for some $ C>0 $. Therefore, using $ \Dehn F((C+1) \cdot \abs{w_2}) $ times relations \cref{relation-f-1}--\cref{relation-f-5}, we can replace $w_2$ by $u$, obtaining the word
	\[
		C_0 \cdot  w_1(S_F) \cdot  C_0 \cdot u(X_1,X_2,Y_1,Y_2).
	\]
	Now using relations \cref{relation-shuffle-one,relation-shuffle-two}, we can shift all letters in $u$ to the left of $C_0$, obtaining
	\[
		C_0 \cdot w_1(S_F) \cdot u(Y_1,Y_2,X_1,X_2) \cdot  C_0,
	\]
	where with $u(Y_1,Y_2,X_1,X_2)$ we denote the word obtained by $ u(X_1,X_2,Y_1,Y_2) $ replacing $ X_1 \rightsquigarrow Y_1 $, $ X_2 \rightsquigarrow Y_2 $, $Y_1 \rightsquigarrow X_1$ and $ Y_2 \rightsquigarrow X_2$ (note that $ C_0 X_1 = Y_1 C_0 $ is conjugate to $ C_0 Y_1 = X_1 C_0 $, since $C_0$ is an involution).
	After conjugating by $C_0$, we get a word in $S_F$. The result follows from the fact that $F$ has quadratic Dehn function \cite{Guba}.
\end{proof}

\begin{proposition}
	There exists $ C>0 $ such that every null-homotopic word of the form
	\[
		w=C_0 \cdot  w_1(S_F) \cdot C_0 \cdot w_2(S_F) \cdot  C_0 \cdot w_3(S_F)
	\]
	has area at most $ C \cdot \abs{w}^2 $.
\end{proposition}

\begin{proof}
	Let $f_2 \in F$ be the element represented by $ w_2 $. If $f_2(\frac12)=\frac12$, then it would follow that $ w(0) = \frac12 $, so $w$ did not represent the identity of $T$. Therefore, assume that $f_2^{-1}(\frac12) > \frac12$ (the other case is equivalent by taking the inverse of $w$), and let $u(Y_1,Y_2) \cdot X_0 \cdot  v(X_1,X_2) $ by the representative for $f_2$ given by \cref{bound-nice-f}.

	We can replace $w_2$ by $u \cdot X_0 \cdot v$ using a number of relations that is at most quadratic in $\abs{w_2}$, since the Dehn function of $F$ is quadratic. We obtain
	\[
		C_0 \cdot w_1(S_F) \cdot C_0 \cdot u(Y_1,Y_2) \cdot X_0 \cdot v(X_1,X_2) \cdot C_0 \cdot  w_3(S_F).
	\]

	Now use relations \cref{relation-shuffle-one,relation-shuffle-two} to shift $u$ to the left of $C_0$, and $v$ to the right of the other $C_0$. We get
	\[
		C_0 \cdot w_1(S_F) \cdot u(X_1, X_2) \cdot C_0 \cdot X_0 \cdot C_0 \cdot v(Y_1,Y_2) \cdot w_3(S_F).
	\]
	Using $ (C_0X_0)^3 = 1 $ we have that $ C_0X_0C_0 = X_0^{-1} C_0 X_0^{-1} $. After this substitution, the word contains only two $C_0$ letters, so we conclude by the previous lemma.
\end{proof}

The proof of \cref{t-dehn-quadratic} is now complete.

\section{Higman-Thompson groups}\label{section-tn}

We now generalize the result to the Higman-Thompson groups $ T_n $. They are defined for $ n\geq 2 $ as the groups of piecewise-linear homeomorphisms of the circle $ [0,1]/\sim $, whose singularities are $n$-adic numbers (that is, of the form $ \frac{a}{n^b} $ for some $ a,b \in \NN $) and with derivative at every other point of the form $ n^k $ for some $ k \in \ZZ $ \cite[Section 4B]{BrownFiniteness}. The subgroup of $ T_{n} $ of homeomorphism that fixes $0$ is denoted by $ F_{n} $. These are generalizations of $ T $ and $F$, since $ T=T_{2} $ and $ F=F_{2} $.

It is convenient to keep in mind the description of elements of $F_n$ as \emph{tree diagrams} \cite{BurilloClearyStein}. Every finite rooted $n$-ary tree is in natural correspondence to a subdivision of $ [0,1] $ in intervals of the form $ [\frac a{n^k}, \frac{a+1}{n^k}] $. Given two finite rooted $n$-ary trees with the same number of leaves, we may construct a unique map $f \in F_n$ that sends the $i$-th interval to the first subdivision to the $i$-th interval of the second subdivision via an affine function. This pair of trees is usually referred to as a \emph{tree diagram} for $f$.

Given a tree diagram, one may choose an integer $i$ and append an $n$-caret (i.e. an $n$-ary tree with only the root and $n$ leaves) to the $i$-th leaf on both trees; doing so does not change the associated map. A tree diagram is called \emph{reduced} if it is not obtained from another tree diagram by applying the previous operation.

The word length of an element $ f \in F_{n} $ can be estimated by using the \emph{number of carets} $ N(f) $ appearing in either tree of the reduced diagram.

\begin{theorem}[{\cite[Theorem 5]{BurilloClearyStein}}]
	\label{bound-via-leaves}
	There is a constant $C>0$ such that for all $ f \in F_{n} $ then
	\[
		\frac1C N(f) \leq \norm{f}_{F_{n}} \leq CN(f).
	\]
\end{theorem}

For $n$-adic rationals $ 0 \leq a < b \leq 1 $, denote by $ F_{n}[a,b] $ the subgroup of $ F_{n} $ that contains all homeomorphism that are the identity outside $ (a,b) $. To ease notation, for $ i \in \range n $ we also denote with $ \Fn[i] $ the subgroup $ F_n[\frac{i-1}{n}, \frac in] $.

\begin{lemma}\label{fn-undistorted-rescaling}
	For every $ i \in \range{n} $, the subgroup $ \Fn[i]$ is isomorphic to $F_n$ and undistorted inside $ F_{n} $.
\end{lemma}

\begin{proof}
	The isomorphism is induced by linearly rescaling $ [0,1] $ to $ [\frac {i-1}n, \frac{i}n] $. This sends the map $ f \in F_n $ with associated tree diagram $(\tree_1,\tree_2)$ to the map $ f' \in \Fn[i] $ with tree diagram $ (\tree_1',\tree_2') $, where for $ j \in \range 2 $ the tree $\tree_j'$ is obtained by attaching the root of $\tree_j$ to the $i$-th leaf of an $n$-caret.

	In particular $ N(f')=N(f) + 1 $, so by \cref{bound-via-leaves} we get that the map above is a quasi-isometric embedding, i.e.~the subgroup is undistorted.
\end{proof}

Denote by $ \tree[i] $ the tree obtained by attaching an $n$-caret to the $i$-th leaf of another $n$-caret. Following \cite[Section 4.11]{BrownFiniteness} define for $ i,j \in \range{n} $ the \emph{glide} $ \gamma_{i,j} \in F_n $ as the element with associated tree diagram $ (\tree[i], \tree[j]) $, and let $ \gamma_i := \gamma_{i,i+1} $ for $ i \in \range{n-1} $.
Let $ \rho \in T_{n} $ be the rotation of order $n$ defined by $ x \mapsto x+\frac 1n \pmod \ZZ $.

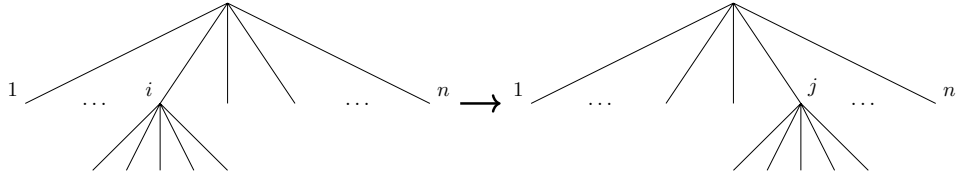
\begin{figure}
	\begin{tikzpicture}[scale=1.7]

	\coordinate (O) at (0,0);
	\coordinate (O') at (5,0);
	\coordinate (I) at (-2/3,-1);
	\coordinate (J') at (17/3,-1);

	\draw
	(O) -- +(-2,-1)
	(O) -- +(-2/3,-1)
	(O) -- +(0,-1)
	(O) -- +(2/3,-1)
	(O) -- +(2,-1);

	\draw
	(O') -- +(-2,-1)
	(O') -- +(-2/3,-1)
	(O') -- +(0,-1)
	(O') -- +(2/3,-1)
	(O') -- +(2,-1);

	\draw
	(I) -- +(-2/3,-2/3)
	(I) -- +(-1/3,-2/3)
	(I) -- +(0,-2/3)
	(I) -- +(1/3,-2/3)
	(I) -- +(2/3,-2/3);

	\draw
	(J') -- +(-2/3,-2/3)
	(J') -- +(-1/3,-2/3)
	(J') -- +(0,-2/3)
	(J') -- +(1/3,-2/3)
	(J') -- +(2/3,-2/3);

	\node[above left] at (-2, -1) {$1$};
	\node[above left] at (3, -1) {$1$};

	\node[above right] at (2, -1) {$n$};
	\node[above right] at (7, -1) {$n$};

	\node[above left] at (I) {$i$};
	\node[above right] at (J') {$j$};

	\node at (-1.3,-1) {$\dots$};
	\node at (+1.3,-1) {$\dots$};

	\node at (3.7,-1) {$\dots$};
	\node at (6.3,-1) {$\dots$};

	\draw[->,very thick] (2.3,-1) -- (2.7, -1);

\end{tikzpicture}
	\caption{The tree diagram for the glide $ \gamma_{i,j} $.}
	\label{tree-diagram-glide}
\end{figure}

\begin{remark}
	When comparing with \cite{BrownFiniteness}, be careful that we are using the convention that maps are applied from right to left, while in \cite{BrownFiniteness} they are applied from left to right.

	Moreover, Brown considers a more general class of groups $ T_{n,r} $, that one may interpret as self-homeomorphisms of the circle described as the interval $ [0,r] / \sim $. Here we are looking in the special case $ n=r $: it is easy to see that $ T_n = T_{n,1} \isom T_{n,n} $ by rescaling.
\end{remark}

\begin{remark}
	In the case $ n=2 $, then $ \gamma_1 = X_0^{-1}$ with respect to the notation in the previous section.
\end{remark}

Both $F_n$ and $T_n$ are finitely presented. We will not need the explicit presentation to compute the Dehn function; below we recall some useful properties.

\begin{lemma}\label{tn-rels}
	The following hold.
	\begin{enumerate}
		\item $ \gamma_{j,k}\gamma_{i,j}$ = $ \gamma_{i,k} $;
		\item $ \rho \gamma_i \rho^{-1} = \gamma_{i+1} $ for $ i \in \range{n-2} $;
		\item $ \rho^{-k} \gamma_{k}^{-1} \rho^{k} = \gamma_{1,n-1} \rho^{-1} \gamma_{1,n-1}$ for $ k \in \range{n-1} $; \label{tn-rel-3}
		\item $ \rho \Fn[i] \rho^{-1} = \Fn[i+1] $ for $ i \in \range{n-1} $;
		\item $ \rho \Fn[n] \rho^{-1} = \Fn[1] $;
		\item $ \gamma_i^{-1} \Fn[i] \gamma_i \subset \Fn[i]$; \label{tn-rel-6}
		\item $ \gamma_i \Fn[i+1]\gamma_i^{-1} \subset \Fn[i+1]$. \label{tn-rel-7}
	\end{enumerate}
\end{lemma}

\begin{proof}
	Most of the computations are straightforward. For \cref{tn-rel-3}, one may either check the equality directly using tree diagrams or obtain it from expression (iv) in \cite[Section 4.11]{BrownFiniteness}, where in our case $ r=n $.

	For \cref{tn-rel-6} and \cref{tn-rel-7} we use that $ \gamma_i^{-1}([\frac{i-1}n, \frac{i}n]) =  [\frac{i-1}n, \frac{i-1}n + \frac1{n^2}] $ and $ \gamma_i([\frac{i}n, \frac{i+1}n]) = [\frac{i+1}n - \frac{1}{n^2}, \frac{i+1}n] $.
\end{proof}

We now generalize the arguments in \cref{section-f} to the case $ n>2 $. The only substantial difference is that we need to be careful as the action of $ F_n $ is not transitive on the $ n $-adic numbers.

\begin{lemma}\label{orbit-fn}
	Let $ x=\frac{a}{n^k}, y=\frac{b}{n^\ell} $ be two $n$-adic numbers in $ (0,1) $. There exists $ f \in F_n $ with $ f(x)=y $ if and only if $ a \equiv b \pmod{n-1}$.

	Moreover, if such an $f$ exists, it can be chosen such that $ \norm f_{F_n} \leq C(k+\ell) $, for some constant $C>0$ independent of $x$ and $y$.
\end{lemma}

\begin{proof}
	We may assume that $a$ and $b$ are not multiple of $n$, as dividing by $n$ does not change the residual class modulo $n-1$.

	Say that an $n$-ary tree contains an $n$-adic rational $q$ if $q$ is an endpoint of the corresponding subdivision. For such a tree, the leaves left (right) of $q$ are the leaves corresponding to intervals that are to the left (right) of $q$.

	We show that if a tree contains $x$, the residual class modulo $n-1$ of the number of leaves left of $x$ is the same as $a$.
	To this end, consider the minimal $n$-ary tree $ \tree_x $ that contains $x$. It is immediate to note that the number of leaves to the left of $x$ is the sum of the digits of $x$ written in base $n$; this has the same residual class modulo $n-1$ as $a$. Every other tree that contains $x$ is obtained by attaching carets to $\tree_x$; since attaching a caret increases the number of leaves by $n-1$, the residual class modulo $n-1$ of the number of leaves left of $x$ is the same as $a$ for every tree containing $x$. Similarly, the number of leaves left to $y$ in a tree containing $y$ is congruent to $ b  $ modulo $n-1$.

	If $f \in F$ sends $x$ to $y$, we can find a tree diagram $(\tree,\tree')$ for $ f $. By possibly attaching carets, we may assume that $ \tree $ contains $x$ and $ \tree' $ contains $y$. Since $x$ is sent to $y$, the number of leaves of $\tree$ left of $x$ coincides with the number of leaves of $\tree'$ left of $y$, and so $ a \equiv b \pmod{n-1} $.

	Vice versa, assume $ a \equiv b \pmod{n-1} $ and consider the minimal trees $ \tree_x $ and $ \tree_y $: they have respectively $ k $ and $ \ell $ carets. Let $L_x$ and $R_x$ be the number of leaves respectively left and right of $x$ in $\tree_x$, and $ L_y $, $ R_y $ be the number of leaves left and right of $y$ in $\tree_y$.

	If $ n-1 $ divides $ a-b $, then it also divides $ L_x-L_y $ by the arguments above. Since $ L_x+R_x \equiv L_y+R_y \equiv 1 \pmod{n-1} $, then $n-1$ also divides $ R_x-R_y $.

	If $ L_x \leq L_y $, we attach $ \frac{L_y-L_x}{n-1} $ carets to the left of $x$ in $T_x$; otherwise, we attach $ \frac{L_x-L_y}{n-1} $ to the left of $y$ in $T_y$. After that, the numbers of leaves to the left of $x$ in $T_x$ and to the left of $y$ in $ \tree_y $ coincide and are equal to
	\[
		\max\set{L_x, L_y} \leq L_x + L_y - 1.
	\]
	We perform the same procedure on the right. In the end, we obtain two trees forming a tree diagram $ (\tree_x', \tree_y') $. The map associated to it sends $x$ to $y$ by construction; the tree diagram has at most $ L_x + R_x - 1 + L_y + R_y -1$ leaves and therefore at most $ k + \ell $ carets, so the estimate follows.
\end{proof}

Let $ S_1 = \set{s_{1,1}, \dots s_{1,\ell}} $ be a finite set of generators for the subgroup $\Fn[1]$, and for $ k \in \range[2]{n} $ let $ S_k = \set{s_{k,1}, \dots, s_{k, \ell}} $ be a generating set for $ \Fn[k] $ defined by $ s_{k,i} = \rho^{k-1} s_{1,i} \rho^{-(k-1)} $.
Let $S_{F_n}$ be $ S_1 \cup \dots S_{n} \cup \set{\gamma_1^{\pm1}, \dots, \gamma_{n-1}^{\pm1}}$, and $ S_{T_n} = S_{F_n} \cup \set{\rho^{\pm1}}$.

\begin{proposition}\label{normal-form-fn}
	The set $ S_{F_n} $ generates $ F_n $. Moreover, there exists a constant $ C > 0 $ such that every element $f \in F_n$ is represented by a word $w(S_{F_n})$ of length at most $ C \cdot \norm f_{F_n} $ with the following property: for every $ k \in \range{n-1} $ there is at most one instance of $ \gamma_k^{\pm 1} $, and $ \gamma_k^{\pm 1} $ appears in $w$ if and only if $ f(\frac kn) \neq \frac kn $.
\end{proposition}

\begin{proof}
	Note that if $ f $ fixes all the points of the form $ \frac kn $, then the result follows from the fact that $ \Fn[1] \times \dots \times \Fn[n] $ is undistorted inside $ \Fn $ by \cref{fn-undistorted-rescaling}.

	Otherwise, let $ k \in \range{n-1}$ be the least such $ f(\frac kn) \neq \frac kn $. It suffices to prove the following claim.

	\emph{Claim.}  We can decompose $f$ as either $ f' \circ \gamma_k \circ f_k $ or $ f_k \circ \gamma_k^{-1} \circ f' $, where $ f' $ fixes $ \frac in $ for $ i \in \range k$ and $ f_k \in \Fn[k] $ with $ \norm {f_k} \leq C \norm f$.

	Indeed, if we assume the claim, we may apply it first on $f$, then on $f'$, and so on, until we obtain a decomposition of $f$ into a product of some $ f_k \in \Fn[k] $ and some $ \gamma_k^{\pm1} $ (in some order), with each $ \gamma_k^{\pm 1} $ appearing at most once. The last part of the statement of the proposition follows from the fact that $ \gamma_k $ is the only generator that does not fix $ \frac kn $.

	Let us prove the claim. By possibly replacing $f$ with its inverse, we may assume that $ f^{-1}(\frac kn) < \frac kn $. Since $ f $ fixes $ \frac{k-1}n $, then it also means that $ f^{-1}(\frac kn) > \frac{k-1}{n} $. We can write $ f^{-1}(\frac kn) $ as $ \frac{k-1}n + \frac d{n^\ell} $.

	By \cref{orbit-fn}, we get that $ d \equiv 1 \pmod{n-1} $. Using the same lemma we may construct $ f_k \in \Fn[k] $ that sends $ f^{-1}(\frac kn) $ to $\frac{k-1}{n} + \frac 1 {n^2}$, and this can be chosen with norm bounded linearly in terms of $\ell$, and therefore in terms of $ \norm f_{F_n} $ by \cref{bound-via-leaves,fn-undistorted-rescaling}.

	So $f' := f \circ f_k^{-1} \circ \gamma_k^{-1} $ fixes the point $ \frac kn $, and clearly also fixes $ \frac in $ for $i<k$. Therefore we get
	\[
		f = f' \circ \gamma_k \circ f_k
	\]
	as required.
\end{proof}

\begin{proposition}\label{normal-form-tn}
	Every $ f \in T_n $ is represented by a word of the form
	\[
		w(S_{F_n}) \cdot \rho^{-k} \cdot \gamma \cdot v(S_{m}),
	\]
	for some $ k \in \range[0]{n-1}$, some $ m \in \range{n} $, and some glide $ \gamma \in F_n $ (possibly the trivial one). Moreover, $w$ and $v$ can be chosen of length at most $ C \cdot \norm f_{T_n} $, for some constant $C>0$ independent of $f$.
\end{proposition}

\begin{proof}
	If $ f \in F_n $, then the result follows from the fact that $ F_n $ is undistorted in $ T_n $ \cite{Sheng}.

	Otherwise, let $ f^{-1}(0) =: \frac{a}n + \frac b{n^\ell} \neq 0 $, with $ \frac b{n^\ell} < \frac1n $, and let $k$ be the unique integer in $ \range{n-1} $ so that $ k \equiv a + b \pmod{n-1} $. If $b=0$ we are done as $ f= f' \circ \rho^{-a} $ for some $ f' \in F_n $, so assume $ b \neq 0 $.

	Let $f_{a+1} \in \Fn[ a+1 ] $ be a function sending $\frac{a}n + \frac b{n^\ell}$ to $ \frac an + \frac{c}{n^2} $, where $ c \in \range{n-1} $ with $ b \equiv c \pmod{n-1} $. We now choose a glide $\gamma$ that sends $ \frac an + \frac{c}{n^2} $ to $ \frac kn $; one possible choice is
	\begin{itemize}
		\item either $ \gamma_{a+1,n} $, if $ a+c = k $,
		\item or $ \gamma_{a+1,1} $ if $ a+c=k+n-1 $.
	\end{itemize}

	This can be checked by looking at tree diagrams for $ \gamma_{k,1} $ and $ \gamma_{k,n} $.

	In conclusion, the map $ f \circ f_{a+1}^{-1} \circ \gamma^{-1} \circ \rho^{k} $ stabilizes the point $0$, so it is equal to some $ f' \in \Fn $.

	So we have obtained
	\[
		f = f' \circ \rho^{-k} \circ \gamma \circ f_{a+1},
	\]
	and we conclude by choosing representatives for the elements $ f_{a+1} \in \Fn[a+1], f' \in \Fn $.

	The estimate follows from the fact that we can choose $f_{a+1}$ with small norm thanks to \cref{orbit-fn}, and from the fact that $ F_n $ is undistorted inside $T_n$ \cite[Theorem 3.7]{Sheng}.
\end{proof}

As done for $T$, since every $ f \in T $ is represented by a word of linear length in $ \norm f_T $ containing at most one power of $ \rho $, by applying \cref{triangles} we only need to check that the words containing at most three powers of $ \rho $ have quadratic area.

\begin{proposition}\label{area-bigons-tn}
	Let \[
		w=\rho^{-k_1} \cdot  w_1(S_{F_n}) \cdot  \rho^{k_2} \cdot  w_2(S_{F_n})
	\]
	be a null-homotopic word, where $ k_1,k_2 \in \range[0]{n-1}$. Then $ k_1=k_2 $ and $ \Area(w) $ is quadratic.
\end{proposition}

\begin{proof}
	If $ k_1=k_2=0 $, the result follows from the fact that $ F_n $ \cite{Zhang} has quadratic Dehn function. Otherwise, both $k_1$ and $k_2 $ must be nonzero, or the resulting word could not have $0$ as a fixed point (and would not be null-homotopic).

	Let $f_1$ be the element of $F_{n}$ represented by $w_1$. Since $w$ is null-homotopic, the map $ \rho^{-k_1} f_1 \rho^{k_2} $ belongs to $ F_n $, so in particular $ f_1(\frac {k_2}n) = \frac{k_1}n $, from which we get by \cref{orbit-fn} that $ k_1=k_2 \eqqcolon k $.

	By \cite{Zhang}, we can replace $w_1$ with a word $w_3$ of the form given by \cref{normal-form-fn} using a quadratic amount of relations. Since $ f_1 $ fixes $ \frac kn $, the word $w_3$ does not contain the generator $ \gamma_k^{\pm 1} $, i.e.~every letter of $w_3$ fixes the point $ \frac kn $.

	So for every letter $s$ of $w_3$, the conjugate $ \rho^{-k} \cdot s \cdot \rho^k $ is still an element of $F_n$. So we can shift the $ \rho^k $ to the left, using that for every letter $s$ in $w_3$ then $ s \rho^k = \rho^k s' $ for some $ s' \in F_n $.

	After using a number of relations that is linear in $w_3$, the powers of $ \rho $ cancel and we are left with a word in $S_{F_n}$ that has quadratic area by \cite{Zhang}.
\end{proof}

\begin{proposition}\label{area-triangles-tn}
	Let \[
		w=\rho^{k_1} \cdot w_1(S_F) \cdot \rho^{k_2} \cdot w_2(S_F) \cdot \rho^{k_3} \cdot w_3 (S_F)
	\]
	be a null-homotopic word, where $ k_1,k_2,k_3 \in \range{n-1} $. Then $ \Area(w) $ is quadratic.
\end{proposition}

\begin{proof}
	By \cref{normal-form-tn}, we can rewrite $ \rho^{k_1} \cdot w_1(S_{F_n}) $ as $ w_4(S_{F_n}) \cdot \rho^{k_1} \cdot \gamma \cdot  w_5(S_{m}) $. This can be done with a quadratic amount of relations by \cref{area-bigons-tn}.

	Now we shift $w_5$ to the right of $ \rho^{k_2} $, obtaining
	\[
		w_4(S_{F_n}) \cdot \rho^{k_1} \cdot \gamma \cdot \rho^{k_2} \cdot w_5(S_{m-k}) \cdot w_2(S_F) \cdot \rho^{k_3} \cdot w_3(S_F),
	\]
	where the index $ m-k $ is to be intended mod $ n $.

	Now by \cref{normal-form-tn} we can replace $ \rho^{k_1} \gamma \rho^{k_2} $, that has length at most $3n$ (a constant), by some word $ w_6(S_{\Fn}) \rho^h w_7(S_{\Fn}) $; the number of relations we need is bounded above by a constant that does not depend on $w$. After doing this, we conclude by the previous lemma.
\end{proof}

We now apply \cref{triangles} where $W_g$ is defined as the words representing $ g \in T_n $ of the form given by \cref{normal-form-tn}; the area of all triangular diagrams is quadratic by \cref{area-bigons-tn,area-triangles-tn}. The proof of \cref{tn-dehn-quadratic} is complete.

\section{Final remarks}\label{section-remarks}

Higman and Brown defined more generally groups $ F_{n,r}, T_{n,r} $ \cite{BrownFiniteness} and $ V_{n,r} $ \cite{Higman} for $ n \geq 2 $ and $ r \geq 1 $. Brown shows that $ F_{n,r} \isom F_n $ for all $ r \geq 1 $, but the same cannot hold for $ T_{n,r} $, for example because the order of torsion elements in $T_n$ satisfies some restrictions \cite{Sheng} while in $T_{n,r}$ there is always a torsion element of order $r$.

It is very likely that the same techniques can be used to prove that the Dehn function of $ T_{n,r} $ is quadratic; however, one should make sure that all the technical lemmas involving distortion still hold in the more general case. In particular, we could not find a reference in the literature proving that $ F_{n,r} $ is undistorted in $ T_{n,r} $.

On the other hand, the situation seems intrinsically more difficult when trying to prove that the Dehn function of $V$ is quadratic. First, unlike what happens for $F$ and $T$, the word metric cannot be estimated as nicely in terms of the reduced tree diagram: the best result one can get \cite{Birget} is that
\[
	\frac 1C N(f) \leq \norm f_V \leq C N(f) \log N(f),
\]
where $N(f)$ denotes the number of carets for a reduced tree diagram for $f$. The second obstacle is that, while elements of $T$ can be expressed by a word containing at most one instance of $C_0$, a similar result cannot hold for $V$, as the number of discontinuities for a function $f \in V$ gives a lower bound for the number of generators not belonging to $T$ in a representative of $f$.

\bibliographystyle{alpha}
\bibliography{references}

\end{document}